\documentclass[a4wide,12pt]{article}
\usepackage[english]{babel}
\usepackage[sectionbib]{natbib}
\usepackage{amsmath}
\usepackage{amsfonts}
\usepackage{amsthm}
\usepackage{amssymb}
\usepackage{graphicx}
\usepackage{dsfont}
\font\tensym=msbm10 \font\sevensym=msbm7 \font\fivesym=msbm5

\newfam\symfam
\textfont\symfam=\tensym \scriptfont\symfam=\sevensym
\scriptscriptfont\symfam=\fivesym

\def\R{{\mathbb{R}}}

\newtheorem{proposition}{Proposition}
\newtheorem{lemma}{Lemma}

\newtheorem{theorem}{Theorem}

\usepackage{epstopdf}
\usepackage[top=3cm , bottom=3.5cm , left=2.5cm ,right=2.5cm ]{geometry}
\usepackage{float}
\usepackage{authblk}

\newcommand{\be}{\begin{eqnarray*}}
    \newcommand{\ee}{\end{eqnarray*}}
\newcommand{\ben}{\begin{eqnarray}}
\newcommand{\een}{\end{eqnarray}}
\begin{document}
\renewcommand{\thefootnote}{\arabic{footnote}}

\begin{center}
{\Large \textbf{Berry-Ess\'een    bound for drift estimation of fractional Ornstein Uhlenbeck process  of second kind }} \\[0pt]
~\\[0pt]
Maoudo Faramba Balde\footnote{%
Cheikh Anta Diop University, Dakar, Senegal.  Email:
\texttt{faramba88@gmail.com}}\footnote{M. F. Balde would like to
acknowledge the NLAGA project of SIMONS foundation and the CEA-MITIC
that partially supported this work.} Rachid Belfadli\footnote{%
 Department of Sciences and Techniques, Cadi Ayyad University, Morocco. Email: \texttt{belfadli@gmail.com}}  Khalifa Es-Sebaiy\footnote{
 Department of Mathematics, Kuwait University, Kuwait. E-mail: \texttt{khalifa.essebaiy@ku.edu.kw}}
 \\[0pt]
~\\[0pt] Cheikh Anta Diop University, Cadi Ayyad University and Kuwait University\\[0pt]
~\\[0pt]
\end{center}
\begin{abstract}
\medskip
In the present paper we consider  the Ornstein-Uhlenbeck process of
the second kind defined as solution to the equation $dX_{t} =
-\alpha X_{t}dt+dY_{t}^{(1)},
   \ \ X_{0}=0$, where $Y_{t}^{(1)}:=\int_{0}^{t}e^{-s}dB^H_{a_{s}}$ with
$a_{t}=He^{\frac{t}{H}}$, and  $B^H$ is a fractional Brownian motion
with Hurst parameter $H\in(\frac12,1)$, whereas    $\alpha>0$ is
unknown parameter to be estimated. We obtain the upper bound
$O(1/\sqrt{T})$ in Kolmogorov distance for normal approximation of
the least squares estimator of the drift parameter $\alpha$ on the
basis of the continuous observation $\{X_t,t\in[0,T]\}$, as
$T\rightarrow\infty$. Our method is based on the work of
\cite{kp-JVA}, which is proved using a combination of Malliavin
calculus and Stein's method for normal approximation.
\end{abstract}
\section{Introduction}
Consider the  fractional Ornstein-Uhlenbeck process (fOU) of the
second kind , defined as the unique pathwise solution to
\begin{equation}
  \label{FOUSK}
  \left\lbrace\begin{aligned}
        &dX_t=-\alpha X_tdt+dY_{t}^{(1)},\quad
t\geq0,\\
      &X_0=0,
    \end{aligned}\right.
\end{equation}
 where
$Y_{t}^{(1)}:=\int_0^t e^{-s}dB^H_{a_s}$  with
$a_{t}=He^{\frac{t}{H}}$, and $B^H := \left\{B^H_t, t \geq
0\right\}$ is a fractional Brownian motion  with Hurst parameter $H
\in ( \frac12, 1)$, whereas    $\alpha>0$ is considered as unknown
parameter.

Let $\widetilde{\alpha}_T$ be the least squares estimator (LSE) for
the parameter $\alpha$,  proposed in the paper \cite{HN}, which is
defined by
\begin{eqnarray}\widetilde{\alpha}_T=\frac{\int_0^TX_tdX_t}{\int_0^TX_t^2dt}=\alpha-\frac{\int_0^TX_tdY_{t}^{(1)}}{\int_0^TX_t^2dt},\label{LSE-FOUSK}
\end{eqnarray}
where the integral with respect to $Y^{(1)}$ is interpreted in the
Skorohod sense.

 \cite{AM} proved that the LSE
$\widetilde{\alpha}_T$ is  consistent and asymptotically normal for
the whole range $H\in(\frac12 , 1)$,  based on the continuous
observation $\{X_t,0\leq t\leq T\}$ as $T\rightarrow\infty$.

However, the study of the asymptotic distribution of an estimator is
not very useful in general for practical purposes unless the rate of
convergence is known. To our knowledge, no result of the
Berry-Ess\'een type is known for the distribution of the LSE
$\widetilde{\alpha}_t$  of the drift parameter $\alpha$ of the fOU
of the second kind (\ref{FOUSK}). The aim of the present work, in
the case $H\in(\frac12 , 1)$, is to provide   an upper bound of
Kolmogorov distance for central limit theorem (CLT) of the LSE
$\widetilde{\alpha}_T$ in the following sense: There exists constant
$0 < C < \infty$, depending only on $\alpha$ and $H$, such that for
all  sufficiently large positive  $T$,
\begin{eqnarray*}
\sup_{z\in \mathbb{R}}\left\vert
P\left(\frac{\sqrt{T}}{\sigma_{\alpha,H}}\left(
\alpha-\widetilde{\alpha}_t\right)\leq z\right)-P\left( Z\leq
z\right)\right\vert\leq \frac{C}{\sqrt{T}},
\end{eqnarray*}
where  $Z$ denotes a standard normal random variable, and the
positive constant $\sigma_{\alpha,H}$ is given by
\begin{eqnarray}
\sigma_{\alpha,H}:=\frac{\alpha}{H\beta(H\alpha+1-H,2H-1)}\sqrt{2\int_{(0,\infty)^3}
F(y_1,y_2,y_3)dy_1dy_2dy_3}<\infty,\label{exp-sigma}
\end{eqnarray}
where $\beta$ denotes the classical Beta function,
$\sigma_{\alpha,H}<\infty$ (see \cite{AV}), and the function $F$ is
defined by
\begin{eqnarray}F(y_1,y_2,y_3):=e^{-\alpha|y_1-y_3|}e^{-\alpha y_2} e^{(1-\frac{1}{H})(y_1+y_2+y_3)} \left|
e^{-\frac{y_2}{H}}-e^{-\frac{y_3}{H}}\right|^{2H-2}\left|
1-e^{-\frac{y_1}{H}}\right|^{2H-2}.\label{fct-F}
\end{eqnarray}

 Let us also describe what is known about this estimation problem in the case
of the Ornstein-Uhlenbeck process of the first kind, defined as
solution to the equation
\begin{eqnarray}
  dX_{t} &=& -\alpha X_{t}dt+dB_{t}^{H},\ \ X_{0}=0,\label{fOU-equ}
\end{eqnarray}
where $\alpha$ is an unknown parameter, and $B^H$ is a fBm with
Hurst parameter $H \in ( 0 , 1)$. The drift parameter estimation
problem  for (\ref{fOU-equ}) observed in continuous time and
discrete time has been studied by using several approaches (see
\cite{KL,HN,HS,BI,EEV,DEV,SV}). In a general case when the process
(\ref{fOU-equ}) is driven by a Gaussian process, \cite{EEO} studied
the non-ergodic case corresponding to $\alpha<0$. They provided
sufficient conditions, based on the properties of the driving
Gaussian process, to ensure that least squares estimators-type of
$\alpha$ are strongly consistent and asymptotically Cauchy. On the
other hand, using  Malliavin calculus advances (see \cite{NP-book}),
\cite{EV} provided new techniques to statistical inference for
stochastic differential equations related to stationary Gaussian
processes, and they used their result to study drift parameter
estimation problems for some stochastic differential equations
driven by fractional Brownian motion with fixed-time-step
observations (in particular for the fOU $X$ given in (\ref{fOU-equ})
with $\alpha>0$). Recently, a Berry-Ess\'een bound of the LSE of the
drift parameter $\alpha>0$ based on the continuous-time observation
of the process (\ref{fOU-equ}) is provided in \cite{CKL} and
\cite{CL} for $H\in[\frac12,\frac34]$ and $H\in(0,\frac12)$,
respectively.

Our article is structured as follows. In section 2, we recall some
basic elements of Malliavin calculus   which are helpful for some of
the arguments we use, and   the  result of \cite{kp-JVA} used in
this paper. In section 3, we provide a rate of convergence to
normality of the LSE $\widetilde{\alpha}_t$ given in
(\ref{LSE-FOUSK}), for any $\frac12<H<1$.

\section{Preliminaries}
In this section, we briefly recall some basic elements of Gaussian
analysis, and Malliavin calculus   which are helpful for some of the
arguments we use. For more details we refer to \cite{NP-book} and
\cite{nualart-book}. We also give here the result of \cite{kp-JVA}
used in  this paper.
\\
Let $B^H = \left\{B^H_t, t \geq 0\right\}$ be a  fractional Brownian
motion (fBm) with Hurst parameter $H \in ( 0 , 1)$ that is     a
centered Gaussian process, defined on a complete probability space
$(\Omega,\mathcal{F}, P)$, with the covariance function
\[R_{H}(t,s)=\frac12\left(t^{2H}+s^{2H}-|t-s|^{2H}\right).\]
Let us now introduce the Gaussian process
$Y_{t}^{(1)}:=\int_{0}^{t}e^{-s}dB^H_{a_{s}},\ t\geq0$,  with
$a_{t}=He^{\frac{t}{H}}$.
\\
Assume that $\frac12<H<1$. Let $f:[0,\infty)\rightarrow\R$  be a
function of class $\mathcal{C}^1$. Then, (see \cite{BEV}),
\begin{eqnarray*}\label{link Y^1 and B}
\int_{s}^{t}f(r)dY^{(1)}_r &=&
\int_{a_s}^{a_t}f(a^{-1}_u)e^{-a^{-1}_u}dB_u,
\end{eqnarray*}where $a^{-1}_u=H\log(u/H)$. Moreover, for every $f,g$
in  $\mathcal{C}^1$,
\begin{eqnarray}
 && E\left(\int_{s}^{t}f(r)dY^{(1)}_r\int_{u}^{v}g(r)dY^{(1)}_r\right) \nonumber\\&=&H(2H-1)\int_{a_s}^{a_t}\int_{a_u}^{a_v}
  f(a^{-1}_x)g(a^{-1}_y)e^{-a^{-1}_x}e^{-a^{-1}_y}|x-y|^{2H-2}dxdy
  \nonumber
\\&=&H^{2H+1}(2H-1)\int_{a_s}^{a_t}\int_{a_u}^{a_v}
  f(H\log(\frac{x}{H}))g(H\log(\frac{y}{H}))(xy)^{-H}|x-y|^{2H-2}dxdy
  \label{inn.scal-1}\\
   &=&  \int_{s}^{t}\int_{u}^{v}f(w)g(z)r_{H}(w,z)dwdz, \label{inn.scal-2}
\end{eqnarray}
where $r_{H}(x,y)$ is a symmetric kernel given by
\begin{eqnarray*}
  r_{H}(w,z) &=& H^{2H-1}(2H-1) (a_wa_z)^{1-H}\left| a_w-a_z\right|^{2H-2}\\
  &=& H^{2H-1}(2H-1)\left(e^{w/H}e^{z/H}\right)^{1-H}\left| e^{w/H}-e^{z/H}\right|^{2H-2}.
\end{eqnarray*}
In particular, we obtain the following covariance given in
\cite{KS},
\begin{eqnarray*}
 \langle
1_{[s,t]},1_{[u,v]}\rangle_{\mathfrak{H}}
=E\left((Y_{t}^{(1)}-Y_{s}^{(1)})(Y_{v}^{(1)}-Y_{u}^{(1)})\right).
\end{eqnarray*}

Let $\mathcal{E}$ denote the space of all real valued step functions
on $\R$. The Hilbert space $\mathfrak{H}$ is defined as the closure
of $\mathcal{E}$ endowed with the inner product
\begin{eqnarray*}
  E\left((Y_{t}^{(1)}-Y_{s}^{(1)})(Y_{v}^{(1)}-Y_{u}^{(1)})\right) &=&
  \int_{s}^{t}\int_{u}^{v}r_{H}(w,z)dwdz.
\end{eqnarray*}
The mapping $1_{[0,t]}\mapsto Y_{t}^{(1)}$  can be extended to a
linear isometry between $\mathfrak{H}$ and the Gaussian space
$\mathcal{H}_1$ spanned by $Y^{(1)}$. We denote this isometry by
$\varphi\in\mathfrak{H}\mapsto Y^{(1)}(\varphi)$.

For a smooth and cylindrical random variable $F =
\left(Y^{(1)}(\varphi_1), \ldots , Y^{(1)}(\varphi_n)\right)$, with
$\varphi_i\in\mathfrak{H},\ i=1,\ldots,n$, and $f\in
\mathcal{C}_b^{\infty} (\R^n)$ ( $f$ and all of its partial
derivatives are bounded), we define its Malliavin derivative as the
$\mathfrak{H}$-valued random variable given by \[DF =
\sum_{i=1}^{n}\frac{\partial f}{\partial
x_i}\left(Y^{(1)}(\varphi_1), \ldots ,
Y^{(1)}(\varphi_n)\right)\varphi_i.\]

 For every $q\geq 1$, ${\mathcal{H}%
}_{q}$ denotes the $q$th Wiener chaos of $Y^{(1)}$, defined as the
closed linear
subspace of $L^{2}(\Omega )$ generated by the random variables $%
\{H_{q}(Y^{(1)}(h)),h\in {{\mathfrak{H}}},\Vert h\Vert _{{\mathfrak{H}}}=1\}$ where $%
H_{q}$ is the $q$th Hermite polynomial.   Wiener chaos of different
orders are orthogonal in $L^{2}\left( \Omega \right) $.
\\
The mapping ${I_{q}(h^{\otimes q}):}=q!H_{q}(Y^{(1)}(h))$ is a
linear isometry between the symmetric tensor product
${\mathfrak{H}}^{\odot q}$ (equipped with the modified norm $\Vert
.\Vert _{{\mathfrak{H}}^{\odot q}}=\sqrt{q!}\Vert .\Vert
_{{\mathfrak{H}}^{\otimes q}}$) and ${\mathcal{H}}_{q}$.
For every $f,g\in {{%
\mathfrak{H}}}^{\odot q}$ the following extended isometry property
holds
\begin{equation*}
E\left( I_{q}(f)I_{q}(g)\right) =q!\langle f,g\rangle _{{\mathfrak{H}}%
^{\otimes q}}.
\end{equation*}%
What is typically referred to as the product formula on Wiener space
is the version of the above formula before taking expectations (see
Section 2.7.3 of \cite{NP-book}). In our work, beyond the zero-order
term in that formula, which coincides with the expectation above, we
will only need to know the full formula for $q=1$, which is
\begin{equation}I_{1}(f)I_{1}(g)=\frac12 I_{2}\left( f\otimes g+g\otimes f\right)
+\langle f,g\rangle
_{{\mathfrak{H}}}.\label{product-formula}\end{equation}

Let $\{e_k , k \geq 1\}$ be a complete orthonormal system in the
Hilbert space ${\mathfrak{H}}$. Given  $f\in {{\mathfrak{H}}}^{\odot
n},g\in {{\mathfrak{H}}}^{\odot m}$, and $p = 1,\ldots , n \wedge
m$, the $p-$th contraction between $f$ and $g$ is the element of
${\mathfrak{H}^{\otimes (m+n-2p)}}$   defined by
 \[f\otimes_p g=\sum_{i_1,\ldots,i_p=1}^{\infty}\langle f,e_{i_1}\otimes\ldots\otimes e_{i_p}\rangle
_{{\mathfrak{H}^{\otimes p}}}\otimes\langle
g,e_{i_1}\otimes\ldots\otimes e_{i_p}\rangle
_{{\mathfrak{H}^{\otimes p}}}.\]

\vspace{0.3cm} Throughout the paper $Z$ denotes a standard normal
random variable, and  $C$ denotes a generic positive constant
(perhaps depending on $\alpha$ and $H$, but not on anything else),
which may change
from line to line.\\

Let us now  state the result of \cite{kp-JVA} we use in this paper.
Recently, based on techniques relied on the combination of Malliavin
calculus and Stein's method (see, e.g., \cite{NP-book}),
\cite{kp-JVA}  provided an  upper bound of the Kolmogrov distance
for central limit theorem of sequences of the form $F_n/G_n$, where
$F_n$ and $G_n$ are functionals of Gaussian fields, see Corollary 1
of \cite{kp-JVA}.

\begin{proposition}[\cite{kp-JVA}]\label{kp}
Let $f_T,g_T\in \mathfrak{H}^{\odot 2}$ for all $T\geq0$, and let
$b_T$ be a positive function of $T$ such that $I_2(g_T)+b_T>0$
almost surely for all $T>0$.  Define for all sufficiently large
positive $T$,
\begin{eqnarray*}
\psi_1(T)&:=&\frac{1}{b_T^2}\sqrt{\left(b^2_T-2\Vert
f_T\Vert^2_{\mathfrak{H}^{\otimes 2}}\right)^2
+8\Vert f_T\otimes_1 f_T\Vert^2_{\mathfrak{H}^{\otimes 2}}},\\
\psi_2(T)&:=&\frac{2}{b_T^2}\sqrt{2 \Vert f_T \otimes_1
g_T\Vert_{\mathfrak{H}^{\otimes2}}+\langle
f_T,g_T\rangle^2_{\mathfrak{H}^{\otimes 2}}},\\
\psi_3(T)&:=&\frac{2}{b_T^2}\sqrt{\Vert
g_T\Vert^4_{\mathfrak{H}^{\otimes 2}} +2\Vert g_T\otimes_1
g_T\Vert^2_{\mathfrak{H}^{\otimes 2}}}.
\end{eqnarray*}
Suppose that $\psi_i(T)\rightarrow 0$ for $i=1,2,3$, as
$T\rightarrow \infty$. Then there exists a positive constant $C$
such that for all sufficiently large positive $T$,
\begin{equation*}
\sup_{z\in \mathbb{R}}\left\vert \mathbb{P}\left(\frac{I_2(f_T)}{
I_2(g_T)+b_T}\leq z\right)-\mathbb{P}\left( Z\leq
z\right)\right\vert\leq C  \max_{i=1,2,3} \psi_i(T).
\end{equation*}
\end{proposition}

\section{Berry Esseen bound for CLT of LSE}

Suppose that  $\frac12<H<1$. Our main interest in this paper is to
provide a Berry-Ess\'een bound for the LSE given in
(\ref{LSE-FOUSK}), of the drift parameter $\alpha>0$ based on the
continuous-time observation of the fOU of the second kind described
by (\ref{FOUSK}).

 Because (\ref{FOUSK}) is linear, it is immediate to solve it explicitly; one then gets
the following formula:
\begin{equation}X_t=e^{-\alpha t}\int_0^te^{\alpha
s}dY_{s}^{(1)}.\label{explicit-FOUSK}
\end{equation}
From (\ref{LSE-FOUSK}) we can write
\begin{eqnarray}\alpha-\widetilde{\alpha}_T=\frac{\int_0^TX_tdY_{t}^{(1)}}{\int_0^TX_t^2dt}.\label{LSE-diff-FOUSK}
\end{eqnarray}
It follows from (\ref{explicit-FOUSK}) that
\begin{eqnarray}\frac{1}{\sqrt{T}}\int_0^TX_tdY_{t}^{(1)}=I^{Y^{(1)}}_2\left(h_T\right),
\mbox{ with }
h_T(s,t):=\frac{1}{2\sqrt{T}}e^{-\alpha|t-s|}1_{[0,T]^2}(s,t).\label{numerator}
\end{eqnarray}
On the other hand, using the product formula
(\ref{product-formula}),
\begin{eqnarray*}X_t^2&=&\left(I_1\left(e^{-\alpha
(t-.)}1_{[0,t]}(.)\right)\right)^2\\
&=&I_2\left(e^{-2\alpha t}e^{\alpha u}e^{\alpha
v}1_{[0,t]^2}(u,v)\right)+\left\|e^{-\alpha
(t-.)}1_{[0,t]}(.)\right\|^2_{\mathfrak{H}}.
\end{eqnarray*}
 Let us introduce the positive constant
\begin{eqnarray}\rho_{\alpha,H}:=\frac{H^{2H}(2H-1)}{\alpha}\beta(H\alpha+1-H,2H-1).\label{exp-rho}\end{eqnarray}
Thus
\begin{eqnarray}&&\frac{1}{T\rho_{\alpha,H}}\int_0^TX_t^2dt\nonumber
\\&=&I_2\left(\frac{1}{T\rho_{\alpha,H}}\int_0^Te^{-2\alpha
t}e^{\alpha u}e^{\alpha
v}1_{[0,t]^2}(u,v)dt\right)+\frac{1}{T\rho_{\alpha,H}}\int_0^Te^{-2\alpha
t}\left\|e^{\alpha u}1_{[0,t]}(u)\right\|^2_{\mathfrak{H}}dt\nonumber\\
&=:&I_2(g_T)+b_T,\label{decomp-denominator}
\end{eqnarray}
where
\begin{eqnarray}b_T:=\frac{1}{T\rho_{\alpha,H}}\int_0^Te^{-2\alpha
t}\left\|e^{\alpha
u}1_{[0,t]}(u)\right\|^2_{\mathfrak{H}}dt,\label{exp-b}
\end{eqnarray}
and
\begin{eqnarray}g_T(u,v)&:=&\frac{1}{T\rho_{\alpha,H}}e^{\alpha
u}e^{\alpha v}\frac{e^{-2\alpha(u\vee v)}-e^{-2\alpha
T}}{2\alpha}1_{[0,T]^2}(u,v)\nonumber\\
&=&\frac{1}{2\alpha\rho_{\alpha,H} T}\left(e^{-\alpha |u-
v|}-e^{-2\alpha T}e^{\alpha u}e^{\alpha
v}\right)1_{[0,T]^2}(u,v)\nonumber
\\
&=&\frac{1}{\alpha\rho_{\alpha,H}
\sqrt{T}}h_T(u,v)-l_T(u,v),\label{decomp-g}
\end{eqnarray}
with $h_T$ is given by (\ref{numerator}), and
\[l_T(u,v):=\frac{1}{2\alpha\rho_{\alpha,H} T}e^{-2\alpha T}e^{\alpha
u}e^{\alpha v}1_{[0,T]^2}(u,v).\] Therefore, combining
(\ref{LSE-diff-FOUSK}), (\ref{numerator}) and
(\ref{decomp-denominator}), we get
\begin{eqnarray}\frac{\sqrt{T}}{\sigma_{\alpha,H}}\left(\alpha-\widetilde{\alpha}_T\right)=\frac{I_2(f_T)}{I_2(g_T)+b_T},\label{LSE2-FOUSK}
\end{eqnarray}
where $\sigma_{\alpha,H}$ is given by (\ref{exp-sigma}), and
\begin{eqnarray}f_T:=\frac{1}{\rho_{\alpha,H}\sigma_{\alpha,H}}h_T.\label{exp-f}
\end{eqnarray}

In order to prove our main result we make use of the following
technical lemmas.

\begin{lemma}\label{lemma1} Let $H\in(\frac12,1)$, and let $b_T$ and  $f_T$ be the functions given by (\ref{exp-b})  and (\ref{exp-f}), respectively.  Then, for all $T>0$,
\begin{eqnarray}|b_T-1|\leq\frac{C}{T},\label{estim-b}
\end{eqnarray}
\begin{eqnarray}
\left|1-2\Vert f_T\Vert^2_{\mathfrak{H}^{\otimes
2}}\right|\leq\frac{C}{T}.\label{estim-f}
\end{eqnarray}
Consequently, for all $T>0$,
\begin{eqnarray*}
\left|b_T^2-2\Vert f_T\Vert^2_{\mathfrak{H}^{\otimes
2}}\right|\leq\frac{C}{T}.
\end{eqnarray*}
\end{lemma}
\begin{proof}
Using (\ref{inn.scal-1}) and making the change of variables $u=x/y$,
we get
\begin{eqnarray}\left\|e^{\alpha
u}1_{[0,t]}(u)\right\|^2_{\mathfrak{H}}&=&H(2H-1)\int_{a_0}^{a_t}\int_{a_0}^{a_t}
 (x/H)^{H\alpha-H}(y/H)^{H\alpha-H}|x-y|^{2H-2}dxdy\nonumber\\
 &=&2H^{2H(1-\alpha)+1}(2H-1)\int_{a_0}^{a_t}dy\int_{a_0}^{y}dx
 (xy)^{H\alpha-H}|x-y|^{2H-2}
 \nonumber\\
 &=&2H^{2H(1-\alpha)+1}(2H-1)\int_{a_0}^{a_t}dyy^{2H\alpha-1}\int_{a_0/y}^{1}du
 u^{H\alpha-H}|1-u|^{2H-2}\nonumber\\
 &=&2H^{2H(1-\alpha)+1}(2H-1)\int_{a_0}^{a_t}dyy^{2H\alpha-1}\int_{0}^{1}du
 u^{H\alpha-H}|1-u|^{2H-2}\nonumber\\
 &&-2H^{2H(1-\alpha)+1}(2H-1)\int_{a_0}^{a_t}dyy^{2H\alpha-1}\int_0^{a_0/y}du
 u^{H\alpha-H}|1-u|^{2H-2}\nonumber\\
 &=:&I_t-J_t,\label{I-J}
\end{eqnarray}
where
\begin{eqnarray*}I_t
 &=&2H^{2H(1-\alpha)+1}(2H-1)\beta(H\alpha+1-H,2H-1)\int_{a_0}^{a_t}y^{2H\alpha-1}dy\\
 &=&\frac{H^{2H}(2H-1)}{\alpha}\beta(H\alpha+1-H,2H-1)(e^{2\alpha
 t}-1).
\end{eqnarray*}
Moreover,
\begin{eqnarray*}\frac{1}{T}\int_0^Te^{-2\alpha t}I_tdt
 &=&\frac{H^{2H}(2H-1)}{\alpha}\beta(H\alpha+1-H,2H-1)\left(1+\frac{e^{-2\alpha
 t}-1}{2\alpha T}\right).
\end{eqnarray*}
Thus
\begin{eqnarray}\left|\frac{1}{T\rho_{\alpha,H}}\int_0^Te^{-2\alpha
t}I_tdt-1
 \right|&=&\frac{1-e^{-2\alpha
 t}}{2\alpha T}
 \leq\frac{1}{2\alpha T}.\label{I}
\end{eqnarray}
On the other hand,
\begin{eqnarray*}|J_t|&\leq&2H^{2H(1-\alpha)+1}(2H-1)\int_{a_0}^{a_t}dyy^{2H\alpha-1}(a_0/y)^{H\alpha}\int_0^{a_0/y}du
 u^{-H}(1-u)^{2H-2}\\&\leq&2H^{2H-\alpha H+1}(2H-1)\int_{a_0}^{a_t}dyy^{H\alpha-1}\int_0^{1}du
 u^{-H}(1-u)^{2H-2}\\
 &=&2H^{2H-\alpha H+1}(2H-1)\beta(1-H,2H-1)\frac{e^{\alpha t}-1}{H\alpha}
 \\&\leq&Ce^{\alpha t}.
\end{eqnarray*}
Hence,
\begin{eqnarray}\frac{1}{T}\int_0^Te^{-2\alpha
t}|J_t|dt&\leq&\frac{C}{T}\int_0^Te^{-\alpha
t}dt\leq\frac{C}{T}.\label{J}
\end{eqnarray}
Therefore, combining (\ref{I-J}), (\ref{I}) and (\ref{J}), we deduce
(\ref{estim-b}).\\
Now let us prove (\ref{estim-f}). First we decompose the integral
$\int_{[0,T]^4}$ into
\begin{eqnarray}\int_{[0,T]^4}=\int_{\cup_{i=1}^{5}A_{i,T}}=\sum_{i=1}^{5}\int_{A_{i,T}},\label{decomp-[0,T]4}\end{eqnarray}
where
\[A_{1,T}=\cup_{i=1}^8D_{i,T}, \ A_{2,T}=\cup_{i=9}^{12}D_{i,T}, \
A_{3,T}=\cup_{i=13}^{16}D_{i,T}, \  A_{4,T}=\cup_{i=17}^{20}D_{i,T},
\ A_{5,T}=\cup_{i=21}^{24}D_{i,T},\] with
\begin{eqnarray*}&&D_{1,T}:=\{0<x_1<x_2<x_3<x_4<T\},\
D_{2,T}:=\{0<x_1<x_2<x_4<x_3<T\},\\
&&D_{3,T}:= \{0<x_2<x_1<x_3<x_4<T,\ D_{4,T}:=
\{0<x_2<x_1<x_4<x_3<T\},
\\
&&D_{5,T}:= \{0<x_3<x_4<x_1<x_2<T\},\ D_{6,T}:=
\{0<x_3<x_4<x_2<x_1<T\},
\\
&&D_{7,T}:= \{0<x_4<x_3<x_1<x_2<T\},\ D_{8,T}:=
\{0<x_4<x_3<x_2<x_1<T\},
\end{eqnarray*}
\begin{eqnarray*}&&D_{9,T}:=\{0<x_1<x_3<x_2<x_4<T\},\
D_{10,T}:=\{0<x_3<x_1<x_4<x_2<T\},\\
&&D_{11,T}:= \{0<x_2<x_4<x_1<x_3<T\},\ D_{12,T}:=
\{0<x_4<x_2<x_3<x_1<T\},
\end{eqnarray*}
\begin{eqnarray*}&&D_{13,T}:=\{0<x_1<x_3<x_4<x_2<T\},\
D_{14,T}:=\{0<x_3<x_1<x_2<x_4<T\},\\
&&D_{15,T}:= \{0<x_2<x_4<x_3<x_1<T\},\ D_{16,T}:=
\{0<x_4<x_2<x_1<x_3<T\},
\end{eqnarray*}
\begin{eqnarray*}&&D_{17,T}:=\{0<x_1<x_4<x_2<x_3<T\},\
D_{18,T}:=\{0<x_4<x_1<x_3<x_2<T\},\\
&&D_{19,T}:= \{0<x_2<x_3<x_1<x_4<T\},\ D_{20,T}:=
\{0<x_3<x_2<x_4<x_1<T\},
\end{eqnarray*}
\begin{eqnarray*}&&D_{21,T}:=\{0<x_1<x_4<x_3<x_2<T\},\
D_{22,T}:=\{0<x_4<x_1<x_2<x_3<T\},\\
&&D_{23,T}:= \{0<x_3<x_2<x_1<x_4<T\},\ D_{24,T}:=
\{0<x_2<x_3<x_4<x_1<T\}.
\end{eqnarray*}
Therefore, using (\ref{inn.scal-2}), (\ref{decomp-[0,T]4}), and
setting
\[m_H(x_1,x_2,x_3,x_4):=e^{-\alpha|x_1-x_3|}e^{-\alpha|x_2-x_4|
}r_H(x_1,x_2)r_H(x_3,x_4),\] we have
\begin{eqnarray}\left\|h_T\right\|^2_{\mathfrak{H}^{\otimes 2}}&=&\frac{1}{4T}\int_{[0,T]^4}m_H(x_1,x_2,x_3,x_4)dx_1\ldots dx_4\nonumber
\\
&=&\frac{1}{4T}\left(\int_{A_{1,T}}+\int_{A_{2,T}}+\int_{A_{3,T}}+\int_{A_{4,T}}+\int_{A_{5,T}}\right)m_H(x_1,x_2,x_3,x_4)dx_1\ldots dx_4\nonumber\\
&=&\frac{1}{4T}\left(8\int_{D_{1,T}}+4\int_{D_{9,T}}+4\int_{D_{13,T}}+4\int_{D_{17,T}}+4\int_{D_{21,T}}\right)m_H(x_1,x_2,x_3,x_4)dx_1\ldots dx_4\nonumber\\
&=:&2I_{1,T}+I_{2,T}+I_{3,T}+I_{4,T}+I_{4,T},\label{estim-h}
\end{eqnarray}
where we used the fact that
\begin{eqnarray*}
\int_{D_{1,T}}m_H(x_1,x_2,x_3,x_4)dx_1dx_2dx_3dx_4=\ldots=\int_{D_{8,T}}
m_H(x_1,x_2,x_3,x_4)dx_1dx_2dx_3dx_4,
\end{eqnarray*}
\begin{eqnarray*}
\int_{D_{9,T}}
m_H(x_1,x_2,x_3,x_4)dx_1dx_2dx_3dx_4=\ldots=\int_{D_{12,T}}
m_H(x_1,x_2,x_3,x_4)dx_1dx_2dx_3dx_4,
\end{eqnarray*}
\begin{eqnarray*}
\int_{D_{13,T}}
m_H(x_1,x_2,x_3,x_4)dx_1dx_2dx_3dx_4=\ldots=\int_{D_{16,T}}
m_H(x_1,x_2,x_3,x_4)dx_1dx_2dx_3dx_4,
\end{eqnarray*}
\begin{eqnarray*}
\int_{D_{17,T}}
m_H(x_1,x_2,x_3,x_4)ddx_1dx_2dx_3dx_4=\ldots=\int_{D_{20,T}}
m_H(x_1,x_2,x_3,x_4)dx_1dx_2dx_3dx_4,
\end{eqnarray*}
\begin{eqnarray*}
\int_{D_{21,T}}
m_H(x_1,x_2,x_3,x_4)dx_1dx_2dx_3dx_4=\ldots=\int_{D_{24,T}}
m_H(x_1,x_2,x_3,x_4)dx_1dx_2dx_3dx_4.
\end{eqnarray*}
Let us now estimate $I_{1,T}$. Making the change of variables
$y_3=x_4-x_1$, $y_2=x_4-x_2$, $y_1=x_4-x_3$,  and $y_4=x_4$, we
obtain
\begin{eqnarray}&&\frac{1}{H^{4H-2}(2H-1)^2}I_{1,T}
\nonumber\\&=&\frac{1}{T}\int_0^{T}dx_4\int_{0<x_1<x_2<x_3<x_4}dx_1dx_2dx_3
e^{-\alpha|x_1-x_3|}e^{-\alpha|x_2-x_4|}
e^{(1/H-1)(x_1+x_2+x_3+x_4)}\nonumber\\&&\times \left|
e^{x_1/H}-e^{x_2/H}\right|^{2H-2}\left|
e^{x_3/H}-e^{x_4/H}\right|^{2H-2}
\nonumber\\&=&\frac{1}{T}\int_0^{T}dy_4\int_{0<y_1<y_2<y_3<y_4}dy_1dy_2dy_3
F(y_1,y_2,y_3)
\nonumber\\&=&\frac{1}{T}\left[\int_0^{T}dy_4\int_{0<y_1<y_2<y_3<\infty}dy_1dy_2dy_3
-\int_0^{T}dy_4\int_{y_4}^{\infty}dy_3\int_0^{y_3}dy_2\int_0^{y_2}dy_1\right]F(y_1,y_2,y_3)
\nonumber\\&=&\left[\int_{0<y_1<y_2<y_3<\infty}dy_1dy_2dy_3
-\frac{1}{T}\int_0^{T}dy_4\int_{y_4}^{\infty}dy_3\int_0^{y_3}dy_2\int_0^{y_2}dy_1\right]F(y_1,y_2,y_3),\label{estim1-D_1}
\end{eqnarray}
where the function $F$ is given by (\ref{fct-F}). Moreover,
\begin{eqnarray}
&&\frac{1}{T}\int_0^{T}dy_4\int_{y_4}^{\infty}dy_3\int_0^{y_3}dy_2\int_0^{y_2}dy_1
F(y_1,y_2,y_3)\nonumber\\
&\leq&\frac{1}{T}\int_0^{T}dy_4\int_{y_4}^{\infty}dy_3\int_0^{y_3}dy_2\int_0^{y_2}dy_1
 e^{-\alpha y_3} e^{(1-\frac{1}{H})(y_1+y_2+y_3)}
\left| e^{-\frac{y_2}{H}}-e^{- \frac{y_3}{H}}\right|^{2H-2}\left|
1-e^{-\frac{y_1}{H}}\right|^{2H-2}
\nonumber\\
&\leq&\frac{H\beta(1-H,2H-1)}{T}\int_0^{T}dy_4\int_{y_4}^{\infty}dy_3\int_0^{y_3}dy_2
 e^{-\alpha y_3} e^{(1-\frac{1}{H})(y_2+y_3)}
\left| e^{-\frac{y_2}{H}}-e^{-\frac{y_3}{H}}\right|^{2H-2}
\nonumber\\
&=&\frac{H\beta(1-H,2H-1)}{T}\int_0^{T}dy_4\int_{y_4}^{\infty}dy_3\int_0^{y_3}dy_2
 e^{-\alpha y_3} e^{(1-\frac{1}{H})(y_3-y_2)}
\left| 1-e^{-(y_3-y_2)/H}\right|^{2H-2}
\nonumber\\
&=&\frac{H\beta(1-H,2H-1)}{T}\int_0^{T}dy_4\int_{y_4}^{\infty}dy_3\int_0^{y_3}dy_2
 e^{-\alpha y_3} e^{(1-\frac{1}{H})x_2}
\left| 1-e^{-\frac{x_2}{H}}\right|^{2H-2}
\nonumber\\
&\leq&\frac{\left(H\beta(1-H,2H-1)\right)^2}{T}\int_0^{T}dy_4\int_{y_4}^{\infty}dy_3
 e^{-\alpha y_3}
\nonumber\\
&\leq&\frac{\left(H\beta(1-H,2H-1)\right)^2}{\alpha^2T}.\label{estim2-D_1}
\end{eqnarray}
Combining (\ref{estim1-D_1}) and (\ref{estim2-D_1}) we deduce
\begin{eqnarray}\left|I_{1,T}-H^{4H-2}(2H-1)^2\int_{0<y_1<y_2<y_3<\infty}dy_1dy_2dy_3
F(y_1,y_2,y_3)\right|\leq \frac{C}{T}.\label{estim1-I_1}
\end{eqnarray}
Moreover,
\begin{eqnarray}\left|I_{1,T}-H^{4H-2}(2H-1)^2\int_{0<y_1<y_3<y_2<\infty}
F(y_1,y_2,y_3)dy_1dy_2dy_3\right|\leq \frac{C}{T},\label{estim2-I_1}
\end{eqnarray}
since $$\int_{0<y_1<y_3<y_2<\infty}
F(y_1,y_2,y_3)dy_1dy_3dy_2=\int_{0<y_1<y_2<y_3<\infty}
F(y_1,y_2,y_3)dy_1dy_2dy_3.$$\\
Using similar arguments as above, we can conclude
\begin{eqnarray}\left|I_{2,T}-H^{4H-2}(2H-1)^2\int_{0<y_2<y_1<y_3<\infty}
F(y_1,y_2,y_3)dy_2dy_1dy_3\right|\leq \frac{C}{T},\label{estim-I_2}
\end{eqnarray}
\begin{eqnarray}\left|I_{3,T}-H^{4H-2}(2H-1)^2\int_{0<y_2<y_3<y_1<\infty}
F(y_1,y_2,y_3)dy_2dy_3dy_1\right|\leq \frac{C}{T},\label{estim-I_3}
\end{eqnarray}
\begin{eqnarray}\left|I_{4,T}-H^{4H-2}(2H-1)^2\int_{0<y_3<y_1<y_2<\infty}
F(y_1,y_2,y_3)dy_3dy_1dy_2\right|\leq \frac{C}{T},\label{estim-I_4}
\end{eqnarray}
\begin{eqnarray}\left|I_{5,T}-H^{4H-2}(2H-1)^2\int_{0<y_3<y_2<y_1<\infty}
F(y_1,y_2,y_3)dy_3dy_2dy_1\right|\leq \frac{C}{T}.\label{estim-I_5}
\end{eqnarray}
Combining (\ref{estim-h}), (\ref{estim1-I_1})---(\ref{estim-I_5})
and the fact that
\begin{eqnarray*}(0,\infty)^3&=&\{0<y_1<y_2<y_3\}\cup\{0<y_1<y_3<y_2\}\cup\{0<y_2<y_1<y_3\}\\
&&\cup\{0<y_2<y_3<y_1\}\cup\{0<y_3<y_1<y_2\}\cup\{0<y_3<y_2<y_1\},
\end{eqnarray*}
 we deduce that
\begin{eqnarray}\left|\left\|h_T\right\|^2_{\mathfrak{H}^{\otimes 2}}-H^{4H-2}(2H-1)^2\int_{(0,\infty)^3}
F(y_1,y_2,y_3)dy_1dy_2dy_3\right|\leq \frac{C}{T},\label{estim2-h}
\end{eqnarray}
which proves (\ref{estim-f}).
\end{proof}

\begin{lemma}\label{lemma2} Suppose $H\in(\frac12,1)$. Let $g_T$ and  $f_T$ be the functions given by (\ref{decomp-g})  and (\ref{exp-f}),
respectively. Then, for all $T>0$,
\begin{eqnarray}
\Vert f_T\otimes_1 f_T\Vert_{\mathfrak{H}^{\otimes 2}}\leq
\frac{C}{\sqrt{T}},\label{f-contract-f}
\end{eqnarray}
\begin{eqnarray}\Vert g_T\Vert_{\mathfrak{H}^{\otimes 2}} \leq \frac{C}{
\sqrt{T}}, \label{norm-g}
\end{eqnarray}
\begin{eqnarray}\Vert g_T\otimes_1 g_T\Vert_{\mathfrak{H}^{\otimes 2}}\leq\frac{C}{
T^{3/2}},\label{g-contract-g}
\end{eqnarray}
\begin{eqnarray}\Vert f_T \otimes_1
g_T\Vert_{\mathfrak{H}^{\otimes2}}\leq \frac{C}{T},
\label{f-contract-g}
\end{eqnarray}
\begin{eqnarray}\left|\langle
f_T,g_T\rangle_{\mathfrak{H}^{\otimes 2}}\right|\leq \frac{C}{
\sqrt{T}}. \label{f-prod-scal-g}
\end{eqnarray}
\end{lemma}
\begin{proof}
Setting $F_T:=I_2(f_T)$, it follows  from Lemma 5.2.4 of
\cite{NP-book} that
\[8\Vert f_T\otimes_1 f_T\Vert^2_{\mathfrak{H}^{\otimes
2}}=Var\left(\frac12\Vert DF_T\Vert^2_{\mathfrak{H}^{\otimes
2}}\right).\] Further, using Lemma 5.1 of \cite{AV}, we have
\[Var\left(\frac12\Vert DF_T\Vert^2_{\mathfrak{H}^{\otimes
2}}\right)\leq \frac{C}{T}.\] Thus the inequality
(\ref{f-contract-f}) is obtained.\\
 Since for every
$(u,v)\in[0,T]^2$, $g_T(u,v)\geq0$, $h_T(u,v)\geq0$ and
$l_T(u,v)\geq0$, then, using (\ref{decomp-g}), we get
\begin{eqnarray*}\Vert g_T\Vert_{\mathfrak{H}^{\otimes 2}}\leq \frac{1}{\alpha\rho_{\alpha,H} \sqrt{T}}\Vert
h_T\Vert_{\mathfrak{H}^{\otimes 2}}.
\end{eqnarray*}
Combining this with (\ref{estim2-h}), we obtain (\ref{norm-g}).
Similarly, using (\ref{decomp-g}), (\ref{exp-f}) and
(\ref{f-contract-f}), we have
\begin{eqnarray*}\Vert g_T\otimes_1
g_T\Vert^2_{\mathfrak{H}^{\otimes 2}}&\leq& \frac{C}{T}\Vert
h_T\otimes_1 h_T\Vert^2_{\mathfrak{H}^{\otimes 2}}\\
&\leq& \frac{C}{T}\Vert f_T\otimes_1
f_T\Vert^2_{\mathfrak{H}^{\otimes
2}}\\
&\leq& \frac{C}{T^3},
\end{eqnarray*} which implies (\ref{g-contract-g}).\\
It is well known that  $$\Vert f_T \otimes_1
g_T\Vert^2_{\mathfrak{H}^{\otimes2}}=\langle
 f_T \otimes_1
f_T, g_T \otimes_1 g_T\rangle_{\mathfrak{H}^{\otimes 2}},$$ due to a
straightforward application of the definition of contractions and
Fubini theorem.\\ Thus, from (\ref{f-contract-f}) and
(\ref{g-contract-g}), we obtain
\begin{eqnarray*}\Vert f_T \otimes_1
g_T\Vert^2_{\mathfrak{H}^{\otimes2}}&\leq& \Vert f_T \otimes_1
f_T\Vert_{\mathfrak{H}^{\otimes2}}\Vert g_T \otimes_1
g_T\Vert_{\mathfrak{H}^{\otimes2}}\\
&\leq&\frac{C}{T^2},
\end{eqnarray*}
which leads to (\ref{f-contract-g}).\\
Finally, the inequality (\ref{f-prod-scal-g}) is a direct
consequence of (\ref{estim-f}) and (\ref{norm-g}).  The proof of the
lemma is thus complete.
\end{proof}

 Our main result is the following theorem. It  is a consequence of
 Proposition \ref{kp},
Lemma \ref{lemma1} and Lemma \ref{lemma2}.
\begin{theorem}
Suppose $H\in(\frac12,1)$. Then, there exists constant $0 < C <
\infty$, depending only on $\alpha$ and $H$, such that for all
sufficiently large positive $T$,
\begin{eqnarray*}
\sup_{z\in \mathbb{R}}\left\vert
P\left(\frac{\sqrt{T}}{\sigma_{\alpha,H}}\left(
\alpha-\widetilde{\alpha}_t\right)\leq z\right)-P\left( Z\leq
z\right)\right\vert\leq \frac{C}{\sqrt{T}},
\end{eqnarray*}
where   $\sigma_{\alpha,H}$ is given by (\ref{exp-sigma}).

\end{theorem}

\renewcommand\bibname

\end{document}